\documentclass{llncs}

\usepackage{makeidx}

\usepackage{amsmath,amssymb}

\begin{document}

\title{An Optimal Control Approach to Herglotz Variational Problems\thanks{This 
is a preprint of a paper whose final and definite form will be published in
\emph{Springer Communications in Computer and Information Science} (CCIS), ISSN	1865-0929. 
Paper submitted 3/Aug/2014; revised 26/Nov/2014; accepted for publication 27/Nov/2014.
Part of first author's Ph.D. project, which is carried out under the 
\emph{Doctoral Programme in Mathematics} (PDMat-UA) of University of Aveiro.}}

\author{Sim\~{a}o P. S. Santos \and Nat\'{a}lia Martins \and Delfim F. M. Torres}

\authorrunning{S. P. S. Santos, N. Martins and D. F. M. Torres}

\institute{CIDMA--Center for Research and Development in Mathematics and Applications,
Department of Mathematics, University of Aveiro, 3810-193 Aveiro, Portugal\\
\email{spsantos@ua.pt, natalia@ua.pt, delfim@ua.pt}}

\maketitle


\begin{abstract}
We address the generalized variational problem of Herglotz from an optimal control point of view.
Using the theory of optimal control, we derive a generalized Euler--Lagrange equation,
a transversality condition, a DuBois--Reymond necessary optimality condition and Noether's
theorem for Herglotz's fundamental problem, valid for piecewise smooth functions.\\

\end{abstract}


\section{Introduction}

The generalized variational problem proposed
by Herglotz in 1930 \cite{Guenther1996,Herglotz1930}
can be formulated as follows:
\begin{equation}
\label{PH_1}
\tag{$P_H$}
\begin{gathered}
z(b)\longrightarrow \text{extr}\\
\text{with }\dot{z}(t)=L(t,x(t),\dot{x}(t),z(t)), \quad t \in [a,b],\\
\text{subject to } x(a)=\alpha, \quad z(a)=\gamma, \quad \alpha, \gamma \in \mathbb{R}.
\end{gathered}
\end{equation}
It consists in the determination of trajectories $x(\cdot)$ and corresponding trajectories
$z(\cdot)$ that extremize (maximize or minimize) the value $z(b)$,
where $L \in C^1([a,b]\times \mathbb{R}^{2n}\times \mathbb{R};\mathbb{R})$.
While in \cite{Guenther1996,Herglotz1930,Simao+NM+Torres2013} the admissible functions
are $x(\cdot) \in C^2([a,b];\mathbb{R}^n)$ and $z(\cdot) \in C^1([a,b];\mathbb{R})$,
here we consider \eqref{PH_1} in the wider class of functions
$x(\cdot) \in PC^1([a,b];\mathbb{R}^n)$ and $z(\cdot) \in PC^1([a,b];\mathbb{R})$.

It is obvious that Herglotz's problem \eqref{PH_1} reduces to the classical fundamental problem
of the calculus of variations (see, e.g., \cite{vanBrunt}) if the Lagrangian $L$
does not depend on the $z$ variable: if
$\dot{z}(t)=L(t,x(t),\dot{x}(t))$, $t \in [a,b]$,
then \eqref{PH_1} is equivalent to the classical variational problem
\begin{equation}
\label{eq:class:Funct}
\int_a^b L(t,x(t),\dot{x}(t))dt \longrightarrow \textrm{extr},
\quad x(a) = \alpha.
\end{equation}
Herglotz proved that an Euler--Lagrange optimality
condition for a pair $\left(x(\cdot),z(\cdot)\right)$ to be an extremizer
of the generalized variational problem \eqref{PH_1} is given by
\begin{multline}
\label{eq:gen:EL}
\frac{\partial L}{\partial x}\left(t,x(t),\dot{x}(t),z(t)\right)
-\frac{d}{dt}\frac{\partial L}{\partial \dot{x}}\left(t,x(t),\dot{x}(t),z(t)\right)\\
+\frac{\partial L}{\partial z}\left(t,x(t),\dot{x}(t),z(t)\right)
\frac{\partial L}{\partial \dot{x}}\left(t,x(t),\dot{x}(t),z(t)\right) = 0,
\end{multline}
$t \in [a,b]$. The equation \eqref{eq:gen:EL} is known as the generalized
Euler--Lagrange equation. Observe that for the fundamental problem
of the calculus of variations \eqref{eq:class:Funct} one has
$\frac{\partial L}{\partial z}=0$ and the differential equation
\eqref{eq:gen:EL} reduces to the classical Euler--Lagrange equation
$$
\frac{\partial L}{\partial x}\left(t,x(t),\dot{x}(t)\right)
-\frac{d}{dt}\frac{\partial L}{\partial \dot{x}}\left(t,x(t),\dot{x}(t)\right) =0.
$$

Since the celebrated work \cite{Pontryagin} by Pontryagin et al.,
the calculus of variations is seen as part of optimal control.
One of the simplest problems of optimal control, in Bolza form,
is the following one:
\begin{equation}
\label{problem P}
\tag{$P$}
\begin{gathered}
\mathcal{J}(x(\cdot),u(\cdot))=\int_a^b f(t,x(t),u(t))dt+\phi(x(b))\longrightarrow\text{extr}\\
\text{subject to } \dot{x}(t)=g(t,x(t),u(t)) \text{ and }  x(a)=\alpha, \quad \alpha \in \mathbb{R},
\end{gathered}
\end{equation}
where $f \in C^1([a,b]\times \mathbb{R}^{n}\times \Omega;\mathbb{R})$,
$\phi \in C^1(\mathbb{R}^{n};\mathbb{R})$,
$g \in C^1([a,b]\times \mathbb{R}^{n}\times \Omega;\mathbb{R}^n)$,
$x \in PC^1([a,b]; \mathbb{R}^n)$ and $u\in PC([a,b];\Omega)$,
with $\Omega \subseteq \mathbb{R}^r$ an open set. In the literature
of optimal control, $x$ and $u$ are called the state
and control variables, respectively, while $\phi$ is known
as the payoff or salvage term. Note that the classical
problem of the calculus of variations \eqref{eq:class:Funct}
is a particular case of problem \eqref{problem P}
with $\phi(x) \equiv 0$, $g(t,x,u)=u$ and $\Omega=\mathbb{R}^n$.
In this work we show how the results on Herglotz's problem
of the calculus of variations \eqref{PH_1} obtained in
\cite{Georgieva2002,Simao+NM+Torres2013}
can be generalized by using the theory of optimal control.
The main idea is simple and consists in rewriting the generalized variational
problem of Herglotz \eqref{PH_1} as a standard optimal control problem \eqref{problem P},
and then to apply available results of optimal control theory.

The paper is organized as follows. In Section~\ref{sec:prelim}
we briefly review the necessary concepts and results from optimal control theory.
In particular, we make use of Pontryagin's maximum principle (Theorem~\ref{PMP});
the DuBois--Reymond condition of optimal control (Theorem~\ref{thm DB-r opt});
and the Noether theorem of optimal control proved in \cite{Torres2001}
(cf. Theorem~\ref{thm opt cont Noether}).
Our contributions are then given in Section~\ref{sec:MainRes}:
we generalize the Euler--Lagrange equation and the transversality condition
for problem \eqref{PH_1} found in \cite{Simao+NM+Torres2013} to admissible
functions $x(\cdot) \in PC^1([a,b];\mathbb{R}^n)$ and $z(\cdot) \in PC^1([a,b];\mathbb{R})$
(Theorem~\ref{Thm:MR1}); we obtain a DuBois--Reymond necessary optimality condition for problem \eqref{PH_1}
(Theorem~\ref{Thm:DR}); and a generalization of the Noether theorem \cite{Georgieva2002}
(Theorem~\ref{thm noether PH1}) as a corollary of the optimal control results of Torres
\cite{Torres2002,Torres2001,Torres2004}.
We end with Section~\ref{sec:conc} of conclusions and future work.


\section{Preliminaries}
\label{sec:prelim}

The central result in optimal control theory is given by Pontryagin's maximum principle,
which is a first-order necessary optimality condition.

\begin{theorem}[Pontryagin's maximum principle for problem \eqref{problem P} \cite{Pontryagin}]
\label{PMP}
If a pair $(x(\cdot),u(\cdot))$ with $x \in PC^1([a,b]; \mathbb{R}^n)$ and $u\in PC([a,b];\Omega)$
is a solution to problem \eqref{problem P}, then there exists $\psi \in PC^1([a,b];\mathbb{R}^n)$
such that the following conditions hold:
\begin{itemize}
\item the optimality condition
\begin{equation}
\label{prob P opt condt}
\frac{\partial H}{\partial u}(t, x(t),u(t), \psi(t))=0;
\end{equation}
\item the adjoint system
\begin{equation}
\label{prob P adj syst}
\begin{cases}
\dot{x}(t)=\frac{\partial H}{\partial \psi}(t, x(t),u(t), \psi(t))\\
\dot{\psi}(t)=-\frac{\partial H}{\partial x}(t, x(t),u(t), \psi(t));
\end{cases}
\end{equation}
\item and the transversality condition
\begin{equation}
\label{prob P tr condt}
\psi(b)=\nabla\phi(x(b));
\end{equation}
\end{itemize}
where the Hamiltonian $H$ is defined by
\begin{equation}
\label{eq:def:Hamiltonian}
H(t,x,u,\psi)=f(t,x,u)+\psi\cdot g(t,x,u).
\end{equation}
\end{theorem}

\begin{definition}[Pontryagin extremal to \eqref{problem P}]
A triplet $(x(\cdot),u(\cdot), \psi(\cdot))$ with
$x \in PC^1([a,b]; \mathbb{R}^n)$, $u\in PC([a,b];\Omega)$ and
$\psi \in PC^1([a,b];\mathbb{R}^n)$ is called a Pontryagin extremal
to problem \eqref{problem P} if it satisfies the optimality condition
\eqref{prob P opt condt}, the adjoint system \eqref{prob P adj syst}
and the transversality condition \eqref{prob P tr condt}.
\end{definition}

\begin{theorem}[DuBois--Reymond condition of optimal control \cite{Pontryagin}]
\label{thm DB-r opt}
If $(x(\cdot),u(\cdot), \psi(\cdot))$ is a Pontryagin extremal
to problem \eqref{problem P}, then the Hamiltonian \eqref{eq:def:Hamiltonian}
satisfies the equality
\begin{equation*}
\frac{d H}{dt}(t,x(t),u(t),\psi(t))=\frac{\partial H}{\partial t}(t,x(t),u(t),\psi(t)),
\end{equation*}
$t \in [a,b]$.
\end{theorem}

Noether's theorem has become a fundamental tool
of modern theoretical physics \cite{MR3072684},
the calculus of variations \cite{MR2099056,MR2098297},
and optimal control \cite{Torres2002,Torres2001,Torres2004}.
It states that when an optimal control problem is invariant
under a one parameter family of transformations,
then there exists a corresponding conservation law:
an expression that is conserved along all the Pontryagin
extremals of the problem \cite{Torres2002,Torres2001,Torres2004,Torres2006}.
Here we use Noether's theorem as found in \cite{Torres2001}, which is formulated
for problems of optimal control in Lagrange form, that is, for problem
\eqref{problem P} with $\phi \equiv 0$. In order to apply the results of \cite{Torres2001}
to the Bolza problem \eqref{problem P}, we rewrite it in the following equivalent Lagrange form:
\begin{equation}
\label{eq:prob:Lag}
\begin{gathered}
\mathcal{I}(x_0(\cdot),x(\cdot),u(\cdot))
=\int_a^b \left[f(t,x(t),u(t))+ x_0(t)\right] dt \longrightarrow\text{extr},\\
\begin{cases}
\dot{x}_0(t)=0,\\
\dot{x}(t)=g\left(t,x(t),u(t)\right),
\end{cases}\\
x_0(a)= \frac{\phi(x(b))}{b-a}, \  x(a)=\alpha.
\end{gathered}
\end{equation}
The notion of invariance for problem \eqref{problem P} is obtained
by applying the notion of invariance found in \cite{Torres2001}
to the equivalent optimal control problem \eqref{eq:prob:Lag}.
In Definition~\ref{def inv OC} we use the little-o notation.

\begin{definition}[Invariance of problem \eqref{problem P}]
\label{def inv OC}
Let $h^s$ be a one-parameter family of $C^1$ invertible maps
\begin{equation*}
\begin{gathered}
h^s:[a,b]\times \mathbb{R}^n\times \Omega \rightarrow\mathbb{R}\times\mathbb{R}^n\times \mathbb{R}^r,\\
h^s(t,x,u)=\left(\mathcal{T}^s(t,x,u), \mathcal{X}^s(t,x,u),\mathcal{U}^s(t,x,u)\right),\\
h^0(t,x,u)=(t,x,u) \text{ for all } (t,x,u)\in [a,b]\times \mathbb{R}^n\times \Omega.
\end{gathered}
\end{equation*}
Problem \eqref{problem P} is said to be invariant under transformations $h^s$
if for all $(x(\cdot),u(\cdot))$ the following two conditions hold:
\begin{enumerate}
\item[(i)]
\begin{multline}
\label{eq inv OC 1}
\left[f \circ h^s(t,x(t),u(t))+\frac{\phi(x(b))}{b-a} + \xi s + o(s)\right]\frac{d\mathcal{T}^s}{dt}(t,x(t),u(t))\\
= f(t,x(t),u(t)) + \frac{\phi(x(b))}{b-a}
\end{multline}
for some constant $\xi$;
\item[(ii)]
\begin{equation}
\label{eq inv OC 2}
\frac{d\mathcal{X}^s}{dt}\left(t,x(t),u(t)\right)
=g\circ h^s(t,x(t),u(t))\frac{d\mathcal{T}^s}{dt}(t,x(t),u(t)).
\end{equation}
\end{enumerate}
\end{definition}

\begin{theorem}[Noether's theorem for the optimal control problem \eqref{problem P}]
\label{thm opt cont Noether}
If problem \eqref{problem P} is invariant
in the sense of Definition~\ref{def inv OC}, then the quantity
\begin{equation*}
(b-t) \xi + \psi(t) \cdot X(t,x(t),u(t))
-\left[H(t,x(t),u(t),\psi(t)) + \frac{\phi(x(b))}{b-a}\right]\cdot T(t,x(t),u(t))
\end{equation*}
is constant in $t$ along every Pontryagin extremal $(x(\cdot),u(\cdot), \psi(\cdot))$
of problem \eqref{problem P}, where
\begin{equation*}
\begin{gathered}
T(t,x(t),u(t))=\frac{\partial \mathcal{T}^s}{\partial s}(t,x(t),u(t))\biggm\vert_{s=0},\\
X(t,x(t),u(t))=\frac{\partial \mathcal{X}^s}{\partial s}(t,x(t),u(t))\biggm\vert_{s=0},
\end{gathered}
\end{equation*}
and $H$ is defined by \eqref{eq:def:Hamiltonian}.
\end{theorem}

\begin{proof}
The result is a simple exercise obtained by applying
the Noether theorem of \cite{Torres2001}
and the Pontryagin maximum principle (Theorem~\ref{PMP})
to the equivalent optimal control problem \eqref{eq:prob:Lag}
(in particular using the adjoint equation corresponding to the multiplier
associated with the state variable $x_0$ and the respective
transversality condition).
\end{proof}


\section{Main Results}
\label{sec:MainRes}

We begin by introducing some basic definitions for the
generalized variational problem of Herglotz \eqref{PH_1}.

\begin{definition}[Admissible pair to problem \eqref{PH_1}]
We say that $(x(\cdot),z(\cdot))$ with $x(\cdot) \in PC^1([a,b];\mathbb{R}^n)$
and $z(\cdot) \in PC^1([a,b];\mathbb{R})$ is an admissible pair to problem
\eqref{PH_1} if it satisfies the equation
\begin{equation*}
\dot{z}(t)=L(t,x(t),\dot{x}(t),z(t)), \quad t \in [a,b],
\end{equation*}
and the initial conditions $x(a)=\alpha$ and $z(a)=\gamma$, $\alpha, \gamma \in \mathbb{R}$.
\end{definition}

\begin{definition}[Extremizer to problem \eqref{PH_1}]
We say that an admissible pair $(x^*(\cdot),z^*(\cdot))$ is an extremizer
to problem \eqref{PH_1} if $z(b)-z^*(b)$ has the same signal for all admissible
pairs $(x(\cdot),z(\cdot))$ that satisfy $\|z-z^* \|_0< \epsilon$ and $\|x-x^* \|_0< \epsilon$
for some positive real $\epsilon$, where $\|y\|_0=\smash{\displaystyle\max_{a\leq t \leq b}}|y(t)|$.
\end{definition}

We now present a necessary condition for a pair $(x(\cdot),z(\cdot))$
to be a solution (extremizer) to problem \eqref{PH_1}. The following result generalizes
\cite{Guenther1996,Herglotz1930,Simao+NM+Torres2013} by considering a more general class of functions.
To simplify notation, we use the operator $\langle\cdot,\cdot\rangle$ defined by
$$
\langle x, z \rangle(t):=(t,x(t),\dot{x}(t),z(t)).
$$
When there is no possibility of ambiguity,
we sometimes suppress arguments.

\begin{theorem}[Euler--Lagrange equation and transversality condition for problem \eqref{PH_1}]
\label{Thm:MR1}
If $(x(\cdot),z(\cdot))$ is an extremizer to problem \eqref{PH_1}, then the Euler--Lagrange equation
\begin{equation}
\label{PH1_EL}
\frac{\partial L}{\partial x}\langle x, z \rangle(t)
-\frac{d}{dt}\left(\frac{\partial L}{\partial \dot{x}}\right)\langle x, z \rangle(t)
+\frac{\partial L}{\partial z}\langle x, z \rangle(t)
\frac{\partial L}{\partial \dot{x}}\langle x, z \rangle(t) = 0
\end{equation}
holds, $t \in [a,b]$. Moreover, the following
transversality condition holds:
\begin{equation}
\label{PH1_tr_cdt}
\frac{\partial L}{\partial \dot{x}}\langle x,z\rangle(b)=0.
\end{equation}
\end{theorem}

\begin{proof}
Observe that Herglotz's problem \eqref{PH_1} is a particular case
of problem \eqref{problem P} obtained by considering $x$ and $z$ as state variables
(two components of one vectorial state variable),
$\dot{x}$ as the control variable $u$,
and by choosing $f\equiv 0$ and $\phi(x,z)=z$.
Note that since $x(t)\in \mathbb{R}^n$,
we have $u(t)\in \mathbb{R}^n$ (i.e., for Herglotz's problem \eqref{PH_1} one has $r=n$).
In this way, the problem of Herglotz, described as an optimal control problem, takes the form
\begin{equation}
\label{PH_1_OC}
\begin{gathered}
z(b) \longrightarrow \text{extr},\\
\begin{cases}
\dot{x}(t)=u(t),\\
\dot{z}(t)=L(t,x(t),u(t),z(t)),
\end{cases}\\
x(a)=\alpha, \  z(a)=\gamma,
\quad \alpha, \gamma \in \mathbb{R}.
\end{gathered}
\end{equation}
It follows from Pontryagin's maximum principle (Theorem~\ref{PMP})
that there exists $\psi_x \in PC^1([a,b];\mathbb{R}^n)$
and $\psi_z \in PC^1([a,b];\mathbb{R})$ such that the following conditions hold:
\begin{itemize}
\item the optimality condition
\begin{equation}
\label{opt_cond PH1}
\frac{\partial H}{\partial u}(t,x(t),u(t),z(t),\psi_x(t),\psi_z(t))=0;
\end{equation}
\item the adjoint system
\begin{equation}
\label{adj_syst PH1}
\begin{cases}
\dot{x}(t)=\frac{\partial H}{\partial \psi_x}(t,x(t),u(t),z(t),\psi_x(t),\psi_z(t))\\
\dot{z}(t)=\frac{\partial H}{\partial \psi_z}(t,x(t),u(t),z(t),\psi_x(t),\psi_z(t))\\
\dot{\psi}_x(t)=-\frac{\partial H}{\partial x}(t,x(t),u(t),z(t),\psi_x(t),\psi_z(t))\\
\dot{\psi}_z(t)=-\frac{\partial H}{\partial z}(t,x(t),u(t),z(t),\psi_x(t),\psi_z(t));
\end{cases}
\end{equation}		
\item and the transversality conditions
\begin{equation}
\label{tr_cond PH1}
\begin{cases}
\psi_x(b)=0,\\
\psi_z(b)=1,
\end{cases}
\end{equation}
\end{itemize}
where the Hamiltonian $H$ is defined by
$$
H(t,x,u,z,\psi_x,\psi_z)=\psi_x \cdot u+\psi_z \cdot L(t,x,u,z).
$$
Observe that the adjoint system \eqref{adj_syst PH1} implies that
\begin{equation}
\label{adj_syst_double}
\begin{cases}
\dot{\psi}_x=-\psi_z\frac{\partial L}{\partial x}\\
\dot{\psi}_z=-\psi_z\frac{\partial L}{\partial z}.
\end{cases}
\end{equation}
This means that $\psi_z$ is solution of a first-order linear differential equation,
which is solved using an integrand factor to find that
$\psi_z=ke^{-\int_a^t\frac{\partial L}{\partial z}d\theta}$ with $k$ a constant.
From the second transversality condition in \eqref{tr_cond PH1}, we obtain that
$k=e^{\int_a^b\frac{\partial L}{\partial z}d\theta}$ and, consequently,
$$
\psi_z=e^{\int_t^b\frac{\partial L}{\partial z}d\theta}.
$$
The optimality condition \eqref{opt_cond PH1} is equivalent to
$\psi_x+\psi_z\frac{\partial L}{\partial u}=0$ and, after derivation, we obtain that
\begin{equation*}
\dot{\psi}_x=-\frac{d}{dt}\left(\psi_z\frac{\partial L}{\partial u}\right)
=-\dot{\psi}_z\frac{\partial L}{\partial u}-\psi_z\frac{d}{dt}\left(\frac{\partial L}{\partial u}\right)
=\psi_z\frac{\partial L}{\partial z}\frac{\partial L}{\partial u}
-\psi_z\frac{d}{dt}\left(\frac{\partial L}{\partial u}\right).
\end{equation*}
Now, comparing with \eqref{adj_syst_double}, we have
$$
-\psi_z\frac{\partial L}{\partial x}
=\psi_z\frac{\partial L}{\partial z}\frac{\partial L}{\partial u}
-\psi_z\frac{d}{dt}\left(\frac{\partial L}{\partial u}\right).
$$
Since $\psi_z(t)\neq 0$ for all $t \in [a,b]$ and $\dot{x}=u$,
we obtain the Euler--Lagrange equation \eqref{PH1_EL}:
$$
\frac{\partial L}{\partial x}-\frac{d}{dt}\left(\frac{\partial L}{\partial \dot{x}}\right)
+\frac{\partial L}{\partial z}\frac{\partial L}{\partial \dot{x}} = 0.
$$
Note that from the optimality condition \eqref{opt_cond PH1} we obtain that
$\psi_x=-\psi_z\frac{\partial L}{\partial u}=-\psi_z\frac{\partial L}{\partial \dot{x}}$,
which together with transversality condition \eqref{tr_cond PH1} for $\psi_x$
leads to the transversality condition \eqref{PH1_tr_cdt}:
\begin{equation*}
\frac{\partial L}{\partial \dot{x}}(b,x(b),\dot{x}(b),z(b))=0.
\end{equation*}
This concludes the proof.
\end{proof}

\begin{definition}[Extremal to problem \eqref{PH_1}]
We say that an admissible pair $(x(\cdot), z(\cdot))$ is an extremal to problem
\eqref{PH_1} if it satisfies the Euler--Lagrange equation \eqref{PH1_EL}
and the transversality condition \eqref{PH1_tr_cdt}.
\end{definition}

\begin{theorem}[DuBois--Reymond condition for problem \eqref{PH_1}]
\label{Thm:DR}
If $(x(\cdot), z(\cdot))$ is an extremal to problem \eqref{PH_1}, then
\begin{equation*}
\frac{d}{dt}\left( -\psi_z(t)\frac{\partial L}{\partial \dot{x}}\langle x, z \rangle(t)
\dot{x}(t)+\psi_z(t) L\langle x, z \rangle(t)\right)
=\psi_z(t)\frac{\partial L}{\partial t}\langle x, z \rangle(t),
\end{equation*}
$t \in [a,b]$, where
$\psi_z(t)=e^{\int_t^b\frac{\partial L}{\partial z}\langle x,z \rangle(\theta)d\theta}$.
\end{theorem}

\begin{proof}
The result follows from Theorem~\ref{thm DB-r opt},
rewriting problem \eqref{PH_1}
as the optimal control problem \eqref{PH_1_OC}.
\end{proof}

We define invariance for \eqref{PH_1} using Definition~\ref{def inv OC}
for the equivalent optimal control problem \eqref{PH_1_OC}.

\begin{definition}[Invariance of problem \eqref{PH_1}]
\label{def inv_PH1}
Let $h^s$ be a one-parameter family of $C^1$ invertible maps
\begin{equation*}
\begin{gathered}
h^s:[a,b]\times \mathbb{R}^n  \times \mathbb{R}
\rightarrow \mathbb{R}\times \mathbb{R}^n  \times \mathbb{R},\\
h^s(t,x(t),z(t))=(\mathcal{T}^s\langle x,z \rangle(t),\mathcal{X}^s\langle x,z \rangle(t),
\mathcal{Z}^s\langle x,z \rangle(t)),\\
h^0(t,x,z)=(t,x,z),
\quad \forall (t,x,z) \in [a,b]\times \mathbb{R}^n  \times \mathbb{R}.
\end{gathered}
\end{equation*}
Problem \eqref{PH_1} is said to be invariant under the transformations $h^s$
if for all admissible pairs $(x(\cdot),z(\cdot))$ the following two conditions hold:
\begin{enumerate}
\item[(i)]
\begin{equation}
\label{eq inv ph1_1}
\left(\frac{z(b)}{b-a}+\xi s + o(s)\right)\frac{d\mathcal{T}^s}{dt}\langle x,z \rangle(t)
=\frac{z(b)}{b-a}
\end{equation}
for some constant $\xi$;
\item[(ii)]
\begin{multline}
\label{eq inv ph1_2}
\frac{d \mathcal{Z}^s}{dt}\langle x,z \rangle(t)\\
= L\left(\mathcal{T}^s\langle x,z \rangle(t),\mathcal{X}^s\langle x,z \rangle(t),
\frac{d\mathcal{X}^s}{d\mathcal{T}^s}\langle x,z \rangle(t),
\mathcal{Z}^s\langle x,z \rangle(t)\right)\frac{d\mathcal{T}^s}{dt}\langle x,z \rangle(t),
\end{multline}
\end{enumerate}
where
$$
\frac{d\mathcal{X}^s}{d\mathcal{T}^s}\langle x,z \rangle(t)
=\frac{\frac{d\mathcal{X}^s}{dt}\langle x,z \rangle(t)}{\frac{d\mathcal{T}^s}{dt}\langle x,z \rangle(t)}.
$$
\end{definition}

Follows the main result of the paper.

\begin{theorem}[Noether's theorem for problem \eqref{PH_1}]
\label{thm noether PH1}
If problem \eqref{PH_1} is invariant in the sense
of Definition~\ref{def inv_PH1}, then the quantity
\begin{multline}
\label{quantity PH1}
\psi_z(t)\biggl[\frac{\partial L}{\partial \dot{x}}\langle x,z \rangle(t) X\langle x,z \rangle(t)
-Z\langle x,z \rangle(t)\\
+\left(L\langle x,z \rangle(t)-\frac{\partial L}{\partial \dot{x}}\langle x,
z \rangle(t) \dot{x}(t)\right)T\langle x,z \rangle(t)\biggr]
\end{multline}
is constant in $t$ along every extremal of problem \eqref{PH_1}, where
\begin{equation*}
\begin{gathered}
T\langle x,z \rangle(t)=\frac{\partial \mathcal{T}^s}{\partial s}\langle x,z \rangle(t)\biggm\vert_{s=0},\\
X\langle x,z \rangle(t)=\frac{\partial \mathcal{X}^s}{\partial s}\langle x,z \rangle(t)\biggm\vert_{s=0},\\
Z\langle x,z \rangle(t)=\frac{\partial \mathcal{Z}^s}{\partial s}\langle x,z \rangle(t)\biggm\vert_{s=0}
\end{gathered}
\end{equation*}
and $\psi_z(t)=e^{\int_t^b\frac{\partial L}{\partial z}\langle x,z \rangle (\theta)d\theta}$.
\end{theorem}

\begin{proof}
As before, we rewrite problem \eqref{PH_1} in the equivalent optimal control form \eqref{PH_1_OC},
where $x$ and $z$ are the state variables and $u$ the control.
We prove that if problem \eqref{PH_1} is invariant in the sense
of Definition~\ref{def inv_PH1}, then \eqref{PH_1_OC} is invariant in the sense of Definition~\ref{def inv OC}.
First, observe that if equation \eqref{eq inv ph1_1} holds, then \eqref{eq inv OC 1} holds
for \eqref{PH_1_OC}: here $f \equiv 0$, $\phi(x,z) = z$ and \eqref{eq inv OC 1} simplifies to
$\left[\frac{z(b)}{b-a} + \xi s + o(s)\right]\frac{d\mathcal{T}^s}{dt}\langle x,z \rangle(t)
= \frac{z(b)}{b-a}$.
Note that the first equation of the control system of problem \eqref{PH_1_OC} ($u(t) = \dot{x}(t)$) defines
$\mathcal{U}^s:=\frac{d\mathcal{X}^s}{d\mathcal{T}^s}$, that is,
\begin{equation}
\label{eq U^s}
\frac{d \mathcal{X}^s}{dt}\langle x,z \rangle(t)
=\mathcal{U}^s\langle x,z \rangle(t)\frac{d\mathcal{T}^s}{dt}\langle x,z \rangle(t).
\end{equation}
Hence, if equation \eqref{eq inv ph1_2} and \eqref{eq U^s} holds, then
there is also invariance of the control system of \eqref{PH_1_OC} in the sense of \eqref{eq inv OC 2}
and consequently problem \eqref{PH_1_OC} is invariant in the sense of Definition~\ref{def inv OC}.
We are now in conditions to apply Theorem~\ref{thm opt cont Noether} to problem \eqref{PH_1_OC}, which
guarantees  that the quantity
\begin{multline*}
(b-t)\xi +
\psi_x(t)\cdot X(t,x(t),u(t),z(t)) + \psi_z(t)\cdot Z(t,x(t),u(t),z(t))\\
-\left(H(t,x(t),u(t),z(t),\psi_x(t),\psi_z(t)) + \frac{z(b)}{b-a}\right) \cdot T(t,x(t),u(t),z(t))
\end{multline*}
is constant in $t$ along every Pontryagin extremal of problem \eqref{PH_1_OC}, where
\begin{equation*}
H(t,x,u,z,\psi_x,\psi_z)=\psi_x u +\psi_z L(t,x,u,z).
\end{equation*}
This means that the quantity
\begin{multline*}
(b-t)\xi +
\psi_x(t) X\langle x,z \rangle(t) + \psi_z(t) Z\langle x,z \rangle(t)\\
-\left(\psi_x(t) \dot{x}(t) +\psi_z(t) L\langle x,z \rangle(t)
+ \frac{z(b)}{b-a}\right)T\langle x,z \rangle(t)
\end{multline*}
is constant in $t$ along all extremals of problem \eqref{PH_1}, where
$$
\psi_x(t)=-\psi_z(t) \frac{\partial L}{\partial u}\langle x,z \rangle(t)
=-\psi_z(t) \frac{\partial L}{\partial \dot{x}}\langle x,z \rangle(t).
$$
Equivalently,
\begin{multline*}
(b-t)\xi - \frac{z(b)}{b-a}T\langle x,z \rangle(t)
-\psi_z(t)\biggl[\frac{\partial L}{\partial \dot{x}}\langle x,z \rangle(t) X\langle x,z \rangle(t)
-Z\langle x,z \rangle(t)\\
+\left(L\langle x,z \rangle(t)-\frac{\partial L}{\partial \dot{x}}\langle x,
z \rangle(t) \dot{x}(t)\right)T\langle x,z \rangle(t)\biggr]
\end{multline*}
is a constant along the extremals.
To conclude the proof, we just need to prove that the quantity
\begin{equation}
\label{eq:constant}
(b-t)\xi - \frac{z(b)}{b-a}T\langle x,z \rangle(t)
\end{equation}
is a constant.
From the invariance condition \eqref{eq inv ph1_1} we know that
\begin{equation*}
\left(z(b)+\xi (b-a) s + o(s)\right)\frac{d\mathcal{T}^s}{dt}\langle x,z \rangle(t) =z(b).
\end{equation*}
Integrating from $a$ to $t$, we conclude that
\begin{multline}
\label{exp:prv:ad}
\left(z(b)+\xi (b-a) s + o(s)\right)\mathcal{T}^s\langle x,z \rangle(t)\\
=z(b)(t-a)+\left(z(b)+\xi (b-a) s + o(s)\right)\mathcal{T}^s\langle x,z \rangle(a).
\end{multline}
Differentiating \eqref{exp:prv:ad} with respect to $s$, and then putting $s=0$, we obtain
\begin{equation}
\label{eq:relation}
\xi (b-a) t + z(b) T\langle x,z \rangle(t) =\xi(b-a)a+z(b)T\langle x,z\rangle(a).
\end{equation}
We conclude from \eqref{eq:relation} that expression \eqref{eq:constant} is the constant
$(b-a)\xi - \frac{z(b)}{b-a}T\langle x,z\rangle(a)$.
\end{proof}


\section{Conclusion}
\label{sec:conc}

We introduced a different approach to the generalized variational principle of Herglotz,
by looking to Herglotz's problem as an optimal control problem.
A Noether type theorem for Herglotz's problem
was first proved by Georgieva and Guenther
in \cite{Georgieva2002}: under the condition of invariance
\begin{equation}
\label{def inv georg}
\frac{d}{ds}\left[L\left(\mathcal{T}^s\langle x,z \rangle(t),\mathcal{X}^s\langle x,z \rangle(t),
\frac{d\mathcal{X}^s}{d\mathcal{T}^s}\langle x,z \rangle(t),z(t)\right)
\frac{d\mathcal{T}^s}{dt}\langle x,z \rangle(t)\right]\bigg\vert_{s=0}=0,
\end{equation}
they obtained
\begin{equation}
\label{quant georg}
\lambda(t) \Biggl[\frac{\partial L}{\partial \dot{x}}\langle x,z \rangle(t) X\langle x,z \rangle(t)
+\left(L\langle x,z \rangle(t)-\frac{\partial L}{\partial \dot{x}}\langle x,z \rangle(t) \dot{x}(t)\right)T\langle x,z \rangle(t)\Biggr],
\end{equation}
where $\lambda(t)=e^{-\int_a^t\frac{\partial L}{\partial z}\langle x,z \rangle (\theta)d\theta}$,
as a conserved quantity along the extremals of problem \eqref{PH_1}.
Our results improve those of \cite{Georgieva2002} in three ways:
(i) we consider a wider class of piecewise admissible functions;
(ii) we consider a more general notion of invariance
whose transformations $\mathcal{T}^s$, $\mathcal{X}^s$
and $\mathcal{Z}^s$ may also depend on velocities, i.e.,
on $\dot{x}(t)$ (note that if \eqref{eq inv ph1_2} holds
with $\mathcal{Z}^s\langle x,z \rangle = z$, then \eqref{def inv georg} also holds);
(iii) the conserved quantity \eqref{quant georg}, up to multiplication by a constant,
is a particular case of \eqref{quantity PH1} when there is no transformation in $z$
($Z=\left.\frac{\partial \mathcal{Z}^s}{\partial s}\right\vert_{s=0}=0$).
The results here obtained can be generalized to higher-order variational problems of Herglotz type.
This is under investigation and will be addressed elsewhere.


\section*{Acknowledgments}

This work was supported by Portuguese funds through the
\emph{Center for Research and Development in Mathematics and Applications} (CIDMA)
and the \emph{Portuguese Foundation for Science and Technology} (FCT),
within project PEst-OE/MAT/UI4106/2014. The authors would like to thank
an anonymous Reviewer for valuable comments.



\end{document}